\documentclass[letterpaper,12pt]{article}
\usepackage[margin=1.2in]{geometry}
\usepackage{enumerate}
\usepackage{enumitem}

\usepackage{amssymb}
\usepackage{amsmath}
\usepackage{amsthm}
\usepackage{mathrsfs}
\usepackage{mathtools}
\usepackage{algorithm}
\usepackage{algorithmicx}
\usepackage{algpseudocode}
\usepackage{bbm}
\usepackage{tikz}
\usepackage{standalone}
\usepackage{import} 
\usepackage{comment}
\usetikzlibrary{calc}
\usepackage{subcaption}
\definecolor{red1}{RGB}{230,25,75}

\usepackage{cleveref}  
\Crefname{graph}{Graph}{Graphs}

\newtheorem{theorem}{Theorem}[section]
\newtheorem{proposition}[theorem]{Proposition}
\newtheorem{corollary}[theorem]{Corollary}
\newtheorem{lemma}[theorem]{Lemma}

\newtheorem{remark}[theorem]{Remark}
\newtheorem{definition}[theorem]{Definition}
\newtheorem{observation}[theorem]{Observation}

\newtheorem{conjecture}[theorem]{Conjecture}
\algnewcommand\algorithmicforeach{\textbf{for each}}
\algdef{S}[FOR]{ForEach}[1]{\algorithmicforeach\ #1\ \algorithmicdo}
\algdef{SE}[REPEATN]{RepeatN}{End}[1]{\algorithmicrepeat\ #1 \textbf{times}}{\algorithmicend}

\usepackage[isbn=false,doi=false,style=numeric,maxbibnames=50]{biblatex}

\begin{document}
\title{On the Number of Path Systems}
\author{Daniel Cizma\thanks{Einstein Institute of Mathematics, Hebrew University, Jerusalem 91904, Israel. e-mail: daniel.cizma@mail.huji.ac.il.} \and  {Nati Linial\thanks{School of Computer Science and Engineering, Hebrew University, Jerusalem 91904, Israel. e-mail: nati@cs.huji.ac.il.{~Supported in part by an ERC Grant 101141253, "Packing in Discrete Domains - Geometry and Analysis".}}}}
\date{}
\maketitle
\begin{abstract}
A {\em path system} in a graph $G$ is a collection of paths, 
with exactly one path between any two vertices in $G$. 
A path system is said to be {\em consistent}
if it is intersection-closed. We show that the number of consistent path systems on $n$ 
vertices is $n^{\frac{n^2}{2}(1-o(1))}$, whereas
the number of consistent path systems which are realizable as the unique geodesics 
w.r.t.\ some metric is only $2^{\Theta(n^2)}$.

In addition, these insights allow us to improve known bounds
on the face-count of the metric cone and shed new light on enumerating
maximum-VC-classes.
\end{abstract}
\section{Introduction}
This paper studies graphs as geometric objects.
{\em Geodesics} offer an important perspective of a geometric
space by specifying how best to move between its points. When the 
space in question is a connected 
undirected
graph $G=(V,E)$, the answer is provided
by a {\em path system} in $G$. To wit,
given any two vertices $u,v \in V$ we 
have a designated path, 
$P_{u,v}=P_{v,u}$, between $u$ 
and $v$. We view $P_{u,v}$
as our `chosen route' between $u$ and $v$. {\em Metric} path systems
are the most well-known examples. Such a path system is defined by
assigning positive weights 
$w: E \to \mathbb{R}_{>0}$ to $G$'s
edges and letting $P_{u,v}$ be
a $w$-shortest $uv$ path. In a
{\em strictly metric} path system, $w$
has the property that the
shortest paths are unique. In addition,
a strictly metric path system is consistent
in the sense that
\begin{equation}\label{eq:def}   
\text{If~}a\in P_{u,v} \text{~then~} P_{u,v}
\text{~is the concatenation of~} P_{u,a} \text{~and~} P_{a,v}.
\end{equation}

\noindent
This condition can be viewed as an axiom in the geometry of path systems.
Any path system that satisfies condition (\ref{eq:def})
is said to be {\em consistent}. We have
discovered earlier on that consistent path systems need
not be metric, and a main objective of the present article is to make this distinction
quantitative, by estimating the number of
consistent path systems on $n$ vertices vs.\ the number
of metric ones. There is already a substantial body of work
dealing with path systems  \cite{GLMY, Bo, CL, CCL,CCL2}.

\subsection{Main Results and some Background}
Among our main discoveries, we show
(Theorem \ref{thm:path_system_count}) 
\[ n^{\frac{n^2}{2}(1-o(1))} \le|\mathscr{M}_n|\le |\mathscr{P}_n| \le n^{\frac{n^2}{2}},\]
where $\mathscr{P}_n$ and
$\mathscr{M}_n$ is the set of all 
consistent resp.\ metric path systems 
on $n$ vertices.\\

In contrast (Theorem \ref{thm:strict_count}) 
\[\text{There are between~}2^{(0.51-o(1))n^2}\text{~and~} 2^{(1.38-o(1))n^2} \text{~strictly metric $n$-vertex path systems.}\] 

Let us mention some pertinent earlier work.
Perhaps most relevant is work by Graham, Yao and Yao \cite{GYY}
who considered similar counting problems. Their hope was that
improved estimates would lead
to stricter lower bounds on the computational complexity of
certain graph metric problems. In particular,
we manage to improve their bound
on the count of faces of the metric cone. An interesting question that we study here
is how many strictly metric systems a fixed graph can have. Our best construction is  
inspired by an idea from \cite{GYY}. Namely,
we introduce the notion of a monotone path system and show that certain classes of monotone path systems are strictly metric.

Perhaps more surprisingly, we discover a relation between consistent path systems
and the theory of VC dimension. A paper of Alon, Moran, and Yehudayoff \cite{AMY}
together with subsequent work of Balogh, Mészáros, and Wagner \cite{BMW}
provide fairly tight bounds on the number of
maximum-VC-classes. Here we place the latter results in the
context of consistent path systems.

Finally, we draw connections between path systems and the notion
of {\em betweenness} from {\em ordered geometry},
e.g., Coxeter's book \cite{Co}. Inspired by the notion of colinear triples in metric spaces,
we introduce the notion of a {\em résumé}. This is a concise depiction of a path system,
which associates every pair of nonadjacent vertices $u,v$ in $G$ to 
{\em some} vertex on the chosen $uv$-path.
Betweenness in metric spaces suggests
a significant perspective that we develop, leading to a characterization
(\Cref{thm:triple_characterization}) of metric and strictly metric
path systems.

\noindent
This is a rich field of study, as manifested by the many open questions that we pose (\Cref{sec:open}).


\section{Consistent Path Systems}
A consistent path system $\mathcal{P}$ in a graph $G=(V,E)$ is a collection of simple paths in $G$ such that:
 \begin{enumerate}[label={\arabic*)}]
\item For every pair $u,v\in V$ there is exactly one $uv$-path in $\mathcal{P}$, 
called $P_{u,v}$. The paths $P_{u,v}$ and $P_{v,u}$ are identical. 
\item The non-empty intersection of any two paths in $\mathcal{P}$ is either a single vertex or a path in $\mathcal{P}$.
 \end{enumerate}
A path system $\mathcal{P}$ in $G=(V,E)$ is {\em neighborly} if 
for every edge $uv\in E$ the $(uv)$-path in $\mathcal{P}$ is this edge.
A path system $\mathcal{P}$ in $G=(V,E)$ is {\em metric} if there is a positive weight function \mbox{$w:E\to \mathbb{R}_{>0}$} such that each path in $\mathcal{P}$ is a $w$-shortest path. Similarly, a path system $\mathcal{P}$ of $G=(V,E)$ is {\em strictly metric} if there is a positive weight function \mbox{$w:E\to \mathbb{R}_{>0}$} such that each path in $\mathcal{P}$ is the {\em unique} $w$-shortest path. \\
A consistent path system $\mathcal{P}$ on a vertex set $V$ is a consistent path system defined on the complete graph on $V$. For $G=(V,E)$:
\begin{center}
	\fontsize{12pt}{20pt}\selectfont
	$\mathscr{P}(G)$  denotes the set of all consistent path systems in $G$.\\
	$\mathscr{M}(G)$  denotes the set of all consistent metric path systems in $G$.\\
	$\mathscr{S}(G)$  denotes the set of all strictly metric path systems in $G$.\\
\end{center}
Clearly,
\[\mathscr{S}(G)\subset\mathscr{M}(G)\subset\mathscr{P}(G),\]
and we know by now that both inclusions 
are proper.\\
We use the shorthand $\mathscr{P}_n=\mathscr{P}(K_n)$, 
$\mathscr{M}_n=\mathscr{M}(K_n)$, $\mathscr{S}_n=\mathscr{S}(K_n)$
for the set of all consistent, metric and strictly metric path systems on $n$ vertices, 
respectively.

Let us seek a concise way to represent a
given consistent path system $\mathcal{P}$ on the vertex set $[n]$.
An obvious but inefficient method is to explicitly write out all $\binom{n}{2}$ paths in
$\mathcal{P}$. But there are better approaches: 
for each pair of vertices $1\le u<v \le n$, rather than specifying the entire path $P_{u,v}$, 
record only the vertex, $z$, which is the immediate successor of $u$ along $P_{u,v}$. 
By consistency, $P_{u,v}$ is the concatenation of the edge $uz$ and $P_{z,v}$. 
More generally, if we associate with every $u,v$ a vertex $z\in P_{u,v}\setminus\{u,v\}$, then 
$P_{u,v}=P_{u,z}P_{z,v}$. 
\begin{definition}
	A résumé of a consistent path system $\mathcal{P}$ over $[n]$ is a {\em partial} function $f:\binom{[n]}{2} \to [n]$, s.t.\
	\begin{enumerate}
		\item $\{u,v\} \in \text{dom}(f)$ if and only if $P_{u,v} \neq uv$
		\item  $f(u,v) \in V(P_{u,v})\setminus \{u,v\}$ for all  $\{u,v\} \in \text{dom}(f)$
	\end{enumerate}
	We denote 
	\[\mathcal{R}(\mathcal{P})=\{f:\binom{[n]}{2}\to [n] : f \text{ is a résumé of } \mathcal{P}\}.\]
\end{definition}
Notice that the résumé is uniquely defined for any path system 
in which every path has length at most $2$. In all other cases multiple résumés exist.
As we observe next, a résumé retains all the information of its corresponding path system.
\begin{observation}\label[observation]{obs:map_system}
	Let $\mathcal{P}$ be a consistent path system on $[n]$, and let 
    $f\in \mathcal{R}(\mathcal{P})$ be any corresponding résumé. 
    The following polynomial time algorithm recovers $\mathcal{P}$ from $f$. 
\end{observation}


\noindent
\begin{enumerate}
    \item Initially, no paths $P_{u,v}$ are defined.
    \item For all $uv\notin \text{dom}(f)$ set $P_{u,v} \coloneqq uv$. 
    \item Repeat $n$ times: If $z=f(u,v)$ and $P_{u,z}$ and $P_{z,v}$ are defined set $P_{u,v} \coloneqq P_{u,z}P_{z,v}$.   
\end{enumerate}
The algorithm relies on the fact that if $P_{u,v}\neq uv$, 
then $P_{u,v}$ can be expressed as the concatenations of the two strictly shorter paths 
$P_{u,f(u,v)}$ and $P_{f(u,v),v}$. Initially, only single edge paths are placed
in $\mathcal{P}$. After the $k$-th iteration of the loop in step 3 all the paths in $\mathcal{P}$ of length at most $k+1$ are recovered.


\section{The Number of Consistent Path Systems}
We now estimate the number of consistent path systems on $n$ vertices. The {\em diameter}
of a path system is the largest length of a path in the system.
We make the following easy observation:

\begin{observation}
Every neighborly path system of diameter $2$ is necessarily consistent. 
\end{observation}
Let $\mathcal{D}_k(G)$ denote the collection of neighborly consistent 
path systems of diameter $k$ in $G$. 

\begin{lemma}\label[lemma]{lem:diam2}
	For any graph $G=(V,E)$ 
	\[|\mathcal{D}_2(G)| =\prod_{uv\notin E} |N(u)\cap N(v)|.\]
\end{lemma}
\begin{proof}
If $\text{diam}(G)\ge 3$, the statement is clearly correct, since both sides 
of the equation vanish.
If $\mathcal{P}\in \mathcal{D}_2(G)$, then for
every pair $u,v\in V$, $P_{u,v}$ 
is either an edge or path of length $2$. Since $\mathcal{P}$ is neighborly,
$|P_{u,v}|=2$ if and only if $u$ and $v$ are not neighbors. But $|N(u)\cap N(v)|$ is 
the number of paths of length-$2$ 
between $u$ and $v$. Moreover,
as just observed,
any neighborly path system of 
diameter $2$ is automatically consistent. Therefore, we may choose any path of length $2$
for every non-adjacent pair to obtain a consistent path system.
The conclusion follows.
\end{proof}

The next lemma speaks about
the number of neighborly diameter-$2$ path systems on Erd\"{o}s-Renyi $G(n,p)$ graphs
using standard concentrations inequalities in this domain.  
Any information that we use throughout about $G(n,p)$ graphs
can be found in \cite{FK}.

\begin{lemma}\label[lemma]{lem:diam2Gnp}
If $G\sim G(n,p)$, where $p \ge \Omega(\frac{1}{\log n})$, then a.a.s.\
	\[\log |\mathcal{D}_2(G)|=(1-p)\frac{n^2}{2} \log (p^2n)\left(1+o(1)\right) .\]
\end{lemma}

This yields:
\begin{theorem}\label[theorem]{thm:path_system_count}
We have  
\[ n^{\frac{n^2}{2}(1-o(1))} \le|\mathscr{M}_n|\le |\mathscr{P}_n| \le n^{\frac{n^2}{2}},\]
where $\mathscr{P}_n$ resp.\
$\mathscr{M}_n$ is the set of all 
consistent resp.\ metric path systems 
on $n$ vertices. 
\end{theorem}
\begin{proof}	
Clearly, $|\mathscr{M}_n|\le |\mathscr{P}_n|$. 
By \cref{obs:map_system}, 
$|\mathscr{P}_n|$ is bounded by
the number of maps from 
$\binom{[n]}{2}$ to $[n]$, 
namely, $n^{\binom{n}{2}}$.\\

On the other hand, by \cref{lem:diam2Gnp} for $G\sim G(n,\frac{1}{\log n})$ there holds
\[n^{\frac{n^2}{2}(1+o(1))} = |\mathcal{D}_2(G)|.\]
Finally, $\mathcal{D}_2(G) \subset \mathscr{M}_n$, because any path system of 
diameter $2$ is induced by assigning a unit weight to every edge.
\end{proof}

\section{Metrizability and Betweenness}\label{sec:s_met}
Arguably, (strictly) metric path systems constitute the most natural instances of consistency in path systems. This means
that the paths in the system can be realized as the (unique) shortest paths w.r.t.\ some 
edge-weighted $K_n$. We offer a straightforward but useful characterization 
of (strict) metrizability in terms of metrics on $V$. \\
We start with a definition:
a {\em pointed triple} $\{a,b;c\}$ with distinct elements $a,b,c\in V$,
is the set $\{a,b,c\}$ where the point $c$ is distinguished from $a,b$.
Let $\binom{V}{3}^{\ast}$ denote the set of pointed triples over $[n]$
and note that $\big|\binom{V}{3}^{\ast}\big| = 3\binom{|V|}{3}$.
We say the (pointed) triple $\{a,b;c\}$ is {\em colinear} w.r.t.\ a metric
$\rho:V\times V\to \mathbb{R}_{>0}$, if
$\rho(a,b) = \rho(a,c) + \rho(c,b)$, i.e., the triangle inequality holds with 
equality for 
$\{a,b;c\}$. Let $T(\rho)$ denote the set of all $\rho$-colinear pointed triples. 
Similarly, for a consistent path system $\mathcal{P}$ let 
$T(\mathcal{P}) = \{\{a,b;c\}\in \binom{V}{3}^{\ast}: c\in P_{a,b}\}$, i.e., $T(\mathcal{P})$ is the collection of pointed triples which are `colinear' w.r.t.\ the paths in $\mathcal{P}$. 
\begin{theorem}\label[theorem]{thm:triple_characterization}
Let $\mathcal{P}$ be 
a path system on vertex set $V$ 
\begin{enumerate}
\item 
$\mathcal{P}$ is strictly metric if and 
only if $T(\mathcal{P})= T(\rho)$ for some metric $\rho$ on $V$.
\item
$\mathcal{P}$ is metric if and only if $T(\mathcal{P}) \subseteq T(\rho)$
for some metric $\rho$ on $V$. 
\end{enumerate}
\end{theorem}

\begin{proof}(\Cref{thm:triple_characterization})
We only prove (1), since the proof for (2) is similar. Let $\mathcal{P}$ be a strictly metric path system that is strictly induced 
by a weight function $w:\binom{V}{2} \to \mathbb{R}_{>0}$. We show that $T(\mathcal{P})= T(\rho)$, where
$\rho$ is the shortest path metric that $w$ induces. Namely, 
$\rho(a,b) = w(P_{a,b})$ for all $a,b\in V$.\\
Indeed, if $\{a,c;b\}\in T(\mathcal{P})$ then $b\in P_{a,c}$. 
In particular $\rho(a,c) = w(P_{a,c})= w(P_{a,b})+w(P_{b,c}) = \rho(a,b)+\rho(b,c)$,
and $\{a,c;b\}\in T(\rho)$.\\
If $\{a,c;b\}\notin T(\mathcal{P})$ then the $(ac)$-path
$P_{a,b}P_{b,c}$ differs from $P_{a,c}$. It follows 
that $\rho(a,c) = w(P_{a,c}) < w(P_{a,b})+w(P_{b,c}) = \rho(a,b)+\rho(b,c)$
and $\{a,c;b\}\notin T(\rho)$.\\
For the reverse implication, let $\rho$ be a metric whose set of
colinear triples coincides with 
$T(\mathcal{P})$. We show that 
$\mathcal{P}$ is the strictly metric path system that is induced by the following weight function: for $a,b\in V$
\begin{equation}\label{def_of_w}
w(ab) = \begin{cases}
\rho(a,b) & P_{a,b} = ab \\
\infty & \text{otherwise}
\end{cases}.\end{equation}
As we show, every path in $\mathcal{P}$ is the unique $w$-geodesic.
To this end let us make
the following simple observation: If $x_0,x_1,\dots,x_m$ are elements
in a metric space $(V, \Delta)$, then the following are equivalent:
\begin{enumerate}[label=(\roman*)]
	\item For all $0\le i_0< i_1\cdots < i_k \le m$, $\Delta(x_{i_0},x_{i_k}) = \sum_{j=1}^{k} \Delta(x_{i_{j-1}},x_{i_{j}})$
	\item For all $0\le i< j < k \le m$, $\Delta(x_i, x_k)=\Delta(x_i, x_j) + \Delta(x_j, x_k)$
\end{enumerate}
    For $a,b\in V$, by (\ref{def_of_w}) and  the above observation 
	\[w(P_{a,b}) = \sum_{uv\in P_{a,b}} w(uv) = \sum_{uv\in P_{a,b}} \rho(u,v)  = \rho(a,b).\]
Let $Q$ be any other $(ab)$-path. Clearly, $w(Q)=\infty$ unless all of $Q$'s edges are in $\mathcal{P}$. Let us turn to $Q$'s vertices. If $V(Q)\subseteq V(P_{a,b})$ but $Q\neq P$, then $E(Q) \not\subseteq E(P_{a,b})$. But then, any edge $e\in E(Q) \setminus E(P_{a,b})$ is, 
by consistency, not in $\mathcal{P}$, so that $w(Q) = \infty$.\\ Assume next that 
$V(Q)\not\subseteq V(P_{a,b})$. As mentioned, by assumption all of $Q$'s edges are in $\mathcal{P}$. Any vertex $c\in V(Q)\setminus V(P_{a,b})$ splits $Q$ into a concatenation of two paths $Q_1 = x_0x_1\dots x_k$ and $Q_2 = y_0\dots y_m$, where $x_0=a$, $x_k=y_0=c$ and $y_m=b$. But then
    \begin{equation*}
w(Q) = w(Q_1) + w(Q_2) 
= \sum_{i=1}^{k}w(x_{i-1}x_i) +\sum_{i=1}^{m}w(y_{i-1}y_i)
=\sum_{i=1}^{k}\rho(x_{i-1},x_i) +\sum_{i=1}^{m}\rho(y_{i-1},y_i).
	\end{equation*}
But since $c\notin P_{a,b}$ the triple $(a,c,b)$ is not $\rho$-colinear, so by the triangle inequality
\begin{equation*}
w(P_{a,b}) = \rho(a,b) < \rho(a,c) + \rho(c,b) 
\le \sum_{i=1}^{k}\rho(x_{i-1},x_i) +\sum_{i=1}^{m}\rho(y_{i-1},y_i)= w(Q),
	\end{equation*}
\end{proof}
As shown in \cite{Bo} and \cite{CL}, the metrizability of a given path system  can be decided in polynomial time using linear programming. In \cite{Bo}, this problem was rephrased as a multi-commodity flow problem with $O(n^3)$ variables and $O(n^4)$ constraints. In \cite{CL}, the relevant LP has exponentially many constraints and is solved efficiently using the ellipsoid method. Here, we offer a simpler linear program, with $O(n^2)$ variables and $O(n^3)$ constraints. 
\begin{corollary}\label[corollary]{cor:simple_LP}
	Let $\mathcal{P}$ be a consistent path system on $V$. For every $a\neq b\in [n]$ define a variable $x_{a,b}=x_{b,a}$, and consider the following system of inequalities:
	\begin{equation*}
		\begin{split}
			 x_{a,c}  + x_{c,b} - x_{a,b}&\ge s \hspace{5mm}a,b,c\in V, \ c\notin P_{a,b}\\
			 x_{a,c} + x_{c,b} -x_{a,b}&= 0 \hspace{5mm}a,b,c\in V, \ c\in P_{a,b}\\
			x_{a,b} &\ge 1 \hspace{5mm}a,b\in[n]
		\end{split}
	\end{equation*}
	When $s=0$ the above system is feasible if and only if $\mathcal{P}$ is metric. When $s=1$ the above system is feasible if and only if $\mathcal{P}$ is strictly metric. 
\end{corollary}
\begin{proof}
    When $s=1$ this system feasible if and only if there exists a metric $\rho$ such that $T(\mathcal{P})=T(\rho)$. Similarly, when $s=0$ this system feasible if and only if there exists a metric $\rho$ such that $T(\mathcal{P})\subseteq T(\rho)$. The claim follows from \cref{thm:triple_characterization}.
\end{proof}
In fact, more can be said about testing whether some path system $\mathcal{P}$ is strictly metric. Recall that a {\em psuedometric} on $V$ is a function $\rho:V\times V\to \mathbb{R}_{+}$ which 
satisfies all of the properties of a metric except that it allows for $\rho(x,y)=0$ 
even if $x\neq y$. The following characterization says that we may drop the positivity assumption when trying to determine if a path system is strictly metric. We emphasize that this is not true for the non-strict case.
	\begin{corollary}
		\label[corollary]{cor:semi_met_strict}
The following are equivalent for a path system $\mathcal{P}$ on $V$.
\begin{enumerate}[label = {(\arabic*)}]
\item $\mathcal{P}$ is strictly metric.
\item There is a metric $\rho$ on $V$ s.t.\ $T(\rho) = T(\mathcal{P})$.
\item The following system of inequalities with variables $x_{a,b}=x_{b,a}$, $a\neq b\in [n]$, is feasible:
		\begin{equation*}
			\begin{split}
				 x_{a,c}  + x_{c,b} - x_{a,b}&>0 \hspace{5mm}a,b,c\in [n], \ c\notin P_{a,b}\\
				 x_{a,c} + x_{c,b} -x_{a,b}&= 0 \hspace{5mm}a,b,c\in [n], \ c\in P_{a,b}
			\end{split}
		\end{equation*}

	\end{enumerate}	
	\end{corollary}
	\begin{proof}
    That (1) and (2) are equivalent is simply a restatement of \cref{thm:triple_characterization}. Moreover, it is clear that (2) implies (3). 
To see that (3) implies (2), we first observe that every feasible solution to the above system of inequalities is non-negative and hence defines a pseudometric $\rho$ with $T(\mathcal{P})=T(\rho)$. Indeed, 
	\[2x_{a,b}=(x_{a,b}+x_{b,c}-x_{a,c})+(x_{b,a}+x_{a,c}-x_{b,c}) \ge 0.\]
We argue that $\rho$ is actually a metric, i.e. $\rho(u,v)>0$ for all $u\neq v$. Suppose that this is not the case and $\rho(u,v)=0$ for some pair $u$ and $v$. Then for any $z\neq u,v$
\[\rho(z,u) \le \rho(z,v) + \rho(v,u) =\rho(z,v) \hspace{8mm} \text{ and } \hspace{8mm} \rho(z,v) \le \rho(z,u) + \rho(u,v) =\rho(z,u).\]
It follows that the above inequalities are in fact equalities, and 
$\{z,u;v\},\{z,v;u\} \in  T(\rho) = T(\mathcal{P})$. 
Consequently, both $v\in P_{z,u}$ and $u\in P_{z,v}$, contrary to the 
assumed consistency.
	\end{proof}

The (semi)metric cone on $n$ points $\text{MET}_n$ comprises all pseudometrics, i.e. 
\[\text{MET}_n = \{x\in \mathbb{R}^{\binom{n}{2}} : x_{a,b}+x_{b,c}-x_{a,c} \ge 0, \  \forall a,b,c\in [n]\}.\]
\begin{equation}\label{relation:equiv}
\text{The condition~} T(\rho) = T(\tilde{\rho}) \text{~defines an equivalence relation~} \rho\sim\tilde{\rho} \text{~on this cone}.
\end{equation}
It is not hard to see that $\sim$-equivalence classes are in a one-to-one correspondence with the faces of $\text{MET}_n$. In particular, \cref{cor:semi_met_strict} shows how to injectively map any strictly metric path system to a face of the metric cone.

We may further ask for what subsets $S\subset \binom{V}{3}^\ast$ there exists some $\rho\in \text{MET}_n$ with $T(\rho)=S$.  
	For $t=\{a,b;c\} \in \binom{V}{3}^\ast$ let us define $\Delta_t \in \mathbb{R}^{\binom{V}{2}}$ as: 
	\[\Delta_t \coloneqq \Delta_{a,b;c} \coloneqq e_{a,c}+e_{c,b}-e_{a,b}.\] 
	Here, $e_{a,b}$ denotes the standard unit vector of the coordinate $\{a,b\}$.
The following characterization, based on LP duality, is used in the sequel:

\begin{lemma}\label[lemma]{lem:semi_met_char}
Given $S\subset \binom{V}{3}^\ast$, no pseudometric $\rho$ satisfies $S=T(\rho)$ if and only 
if there exist $\alpha_{t} \ge 0$, $t \in \binom{V}{3}^\ast$, such that
\[\sum_{s \in S}\Delta_{s} = \sum_{t \in \binom{V}{3}^\ast} \alpha_t\Delta_{t},\]
and $\text{supp}(\alpha)\not\subseteq S$.
\end{lemma}
\begin{proof}
Note that $\langle \Delta_t, \rho\rangle\geq 0$ for any pseudometric $\rho$ and any 
$t\in \binom{V}{3}^{\ast}$, with equality iff $t\in T(\rho)$.\\
We show first that the existence of such $\alpha$ implies 
that $S$ is not of the form $T(\rho)$.
Our assumption is that $\sum_{s \in S}\Delta_{s} = \sum_{t} \alpha_t\Delta_{t}$ where 
$\alpha_t \ge 0$ and $\text{supp}(\alpha)\not\subseteq S$. Now let 
$\rho\in \mathbb{R}^{\binom{V}{2}}$ be any pseudometric with $S\subseteq T(\rho)$, 
and let us show that $S\neq T(\rho)$, i.e., the inclusion is proper. Indeed,  
\[0 = \sum_{s \in S}\langle\Delta_{s},\rho\rangle =\left\langle\sum_{s \in S}\Delta_{s} ,\rho \right\rangle= \left\langle\sum_{t \in \binom{V}{3}^\ast} \alpha_t\Delta_{t} ,\rho \right\rangle=\sum_{t \in \binom{V}{3}^\ast} \alpha_t\langle\Delta_{t},\rho\rangle.\]
In particular, if $\alpha_t > 0$ then $\langle\Delta_{t},\rho\rangle=0$ and $t\in T(\rho)$. But by assumption, $\text{supp}(\alpha)\not\subseteq S$, hence $S\subsetneq T(\rho)$.\\
In the reverse direction, let us consider a collection $S\subset \binom{V}{3}^\ast$ that is not realizable as the set of colinear triples of a pseudometric. This means that no $x\in \mathbb{R}^{\binom{V}{2}}$ satisfies the following system of inequalities:
		 \begin{equation*}
			\begin{split}
				\langle \Delta_s,x\rangle =0 & \hspace{10mm}\forall s\in S \\
				\langle \Delta_t,x\rangle >0 & \hspace{10mm}\forall t\in S^c \\
			\end{split}
		 \end{equation*}
By LP-duality, 
(e.g., \cite{SW}), 
some linear combination of the above inequalities yields $0<0$. In other words, there exist
real $\{\beta_s |s\in S\}$, and $\{\lambda_t\ge 0|t\in S^c\}$, not all zero, so that 
		 \[\sum_{s\in S} \beta_s \Delta_s - \sum_{t\in S^c} \lambda_t \Delta_t=0.\]
By rescaling, if necessary, we can assume that all
$\beta_s\le 1$, and rewrite,
\[\sum_{s\in S}\Delta_s = \sum_{s\in S} \beta_s \Delta_s + \sum_{s\in S} (1-\beta_s)\Delta_s = \sum_{t\in S^c} \lambda_t \Delta_t + \sum_{s\in S} (1-\beta_s)\Delta_s =  \sum_{t \in \binom{V}{3}^\ast} \alpha_t\Delta_{t},\]
where $\alpha_t = \lambda_t$, $t\notin S$ and $\alpha_t = 1-\beta_t$ for $t\in S$ are all non-negative and at least one $\alpha_t$, $t\notin S$, is non-zero.
	\end{proof}
\subsection{Solution Need Not Be Integral}

It is also possible to view a pointed triple $\{a,b;c\}$ as an oriented triangle, 
where edges $\{a,c\}, \{c,b\}$ are positively oriented and $\{a,b\}$ is negatively oriented.
The corresponding notion of $\{a,b;c\}$'s {\em boundary} is: 
\[\partial(\{a,b;c\}) \coloneqq e_{a,c}+e_{c,b}-e_{a,b} = \Delta_{a,b;c}.\]
We can interpret \cref{lem:semi_met_char} as saying that a subset 
$S\subseteq \binom{[n]}{3}^\ast$ cannot be the collection of colinear
triples of some pseudometric if and only if 
there exists another (fractional) set of pointed triples whose boundary coincides with the 
boundary of $S$. This formulation is very similar to Bodwin's \cite{Bo} perspective, where
strictly metric path systems are presented in terms of multi-commodity 
network flows. 

Famously, network flows satisfy an integrality condition, but this is
no longer the case when we study multi-commodity 
network flows. This prompts us to ask 
whether there always exist an {\em integral} witness for non-metrizability. 
Namely, if the set $S\subseteq \binom{[n]}{3}^\ast$ is not the triple set of any pseudometric. Does there exist a multi-set $T\subseteq \binom{[n]}{3}^\ast$, with $T\setminus S \neq \emptyset$ and $\sum_{s\in S}\Delta_s = \sum_{t\in T}\Delta_t $? The answer to this question is no. The following set 
	\[S = \{\{0, 1; 7\},\{0, 3; 7\}, \{0, 4; 6\}, \{1, 3; 6\} , \{2, 4; 7\}, \{2, 5; 6\}, \{4, 5; 7\}\}\]
	is not the triple system of any metric since
	\begin{equation*}
		\begin{split}
			\sum_{s\in S}\Delta_s &= \frac{1}{8}\Delta_{0,1;6} + \frac{7}{8}\Delta_{0,1;7}+ \frac{1}{8}\Delta_{0,3;6}+ \frac{7}{8}\Delta_{0,3;7}\\ 
			&+ \frac{3}{4}\Delta_{0,4;6}
			+ \frac{1}{4}\Delta_{0,4;7}+ \frac{7}{8}\Delta_{1,3;6}+ \frac{1}{8}\Delta_{1,3;7}\\
			&+ \frac{1}{8}\Delta_{2,4;6}
			+ \frac{7}{8}\Delta_{2,4;7}
			+ \frac{7}{8}\Delta_{2,5;6}+ \frac{1}{8}\Delta_{2,5;7}\\
			&+ \frac{1}{8}\Delta_{4,5;6}+ \frac{7}{8}\Delta_{4,5;7}.
		\end{split}
	\end{equation*}
	On the other hand, there exists no integer witness of this fact. This can be verified by using any ILP solver, e.g., \cite{gurobi}.

\section{Strictly Metric Path Systems}
We turn to estimate the number of strictly metric path systems on $n$ vertices. While the number of consistent (as well as metric) path systems on $n$ vertices is $n^{\frac{n^2}{2}(1-o(1))}$, the number of strictly metric systems turns out to be considerably smaller. 
We prove
\begin{theorem}\label[theorem]{thm:strict_count}
The set $\mathscr{S}_n$ of all strictly metric path systems on $n$ vertices satisfies
\[2^{0.51n^2(1-o(1))}\le|\mathscr{S}_n| \le 2^{1.38n^2(1-o(1))}.\] 
\end{theorem}
A lower bound of $2^{\frac{n^2}{2}(1-o(1))}$ is easy to
prove, since every connected graph $G=(V,E)$
induces at least one strictly metric path system, by assigning
a weight $w(e):=1+\varepsilon_e$
to every edge $e\in E$, where $\varepsilon_e$ is taken to be
a small generic value. 

To improve this lower bound, we consider path systems in $G(n,p)$ random graphs.
While \cref{thm:strict_count} provides only a very sligt improvement
over the trivial lower bound,
its proof shows that almost all graphs admit 
$2^{\Omega(n^2)}$ strictly metric {\em neighborly} path systems, see below.
We denote by $\mathscr{NS}(G)$
the family of neighborly strictly metric
path systems in the graph $G$. The reader may prefer to read first
\cref{thm:bip_strict} and its proof that are similar but easier.
\begin{lemma}\label[lemma]{lem:gnp_systems}
For fixed $0<p<1$ and $G\sim G(n,p)$  
there holds a.a.s.\ 
\[|\mathscr{NS}(G)|\ge 2^{(1-p)p^2\frac{n^2}{8}(1+o(1))}.\]
\end{lemma}
\begin{proof}
For fixed $p$, a $G(n,p)$ graph
has a.a.s.\ a perfect matching $M$.
Let $M$'s edges be called $x_1y_1,\dots, x_ky_k$,
where $k = \lfloor \frac{n}{2}\rfloor$.
We make the edges in $M$ `cheap'
and construct a corresponding family of 
neighborly path systems in $G$, where
shortest paths naturally tend to 
utilize $M$-edges. \\
For $1\le i < j \le k$, we call 
$x_i, x_j$ an {\em admissible pair of vertices} if:
	\begin{enumerate}
		\item $x_i$ and $x_j$ are not neighbors in $G$ 
		\item Both $x_iy_j$ and $x_jy_i$ are edges in $G$ 
	\end{enumerate}
The probability that the vertices
$x_i, x_j$ form an admissible pair is 
$(1-p)p^2$. Also, the events 
"$x_i, x_j$ form an admissible pair" 
are independent over all pairs 
$1\le i<j\le k$. 
By a standard concentration argument, 
the number of admissible pairs is 
a.a.s.\ 
$(1-p)p^2\frac{n^2}{8}(1+o(1))$. 
For each admissible pair $x_i, x_j$ 
there are precisely two 
$2$-step $(x_ix_j)$-paths which use an $M$-edge, namely
$x_iy_ix_j$ and $x_iy_jx_j$.
Clearly, there are 
$2^{(1-p)p^2\frac{n^2}{8}(1+o(1))}$
partial path systems where every
admissible pair is connected by
one of these two paths. We now argue that
these systems are strictly metric. Fix one such partial system $\tilde{\mathcal{P}}$, and
let $X_{u,v}$ be i.i.d.\ random variables over $G$'s edges $uv$. Every such a random
variable is uniformly distributed on $[0,\epsilon]$, with $\epsilon>0$ some small constant. 
The value of $X_{u,v}$ gets added to the weight of this edge $uv$. This
random noise ensures that every shortest path is unique. Concretely, here is the
weight function $w:E\to \mathbb{R}_{>0}$ that we use:
	\[w(uv) \coloneqq \begin{cases}
		1 + X_{u,v}& uv\in M \\
		1.1+ X_{u,v} & uv \notin M \text{ and } uvz\in \tilde{\mathcal{P}} \text{ for some } z\\
		1.2 + X_{u,v} & \text{Otherwise}
	\end{cases}\]
In words, we assign edge weights
as follows: $M$-edges are the cheapest. 
Next come edges not in $M$ that
are in one of the chosen paths in 
$\tilde{\mathcal{P}}$.
All other edges are more expensive.
By construction, an edge not in $M$ 
appears in at most one path in $\tilde{\mathcal{P}}$. It is not difficult to see that every path in $\tilde{\mathcal{P}}$ is the unique 
$w$-shortest path between its ends. 
Moreover, the random 
noise makes all the shortest paths unique. Therefore, $w$ strictly induces a path system $\mathcal{P}$ containing all the paths in $\tilde{\mathcal{P}}$.
\end{proof}

To prove the upper bound in \cref{thm:strict_count} we use the insight gained in 
\cref{cor:semi_met_strict}. In particular, we use \cref{lem:semi_met_char} to bound the number of $\sim$-equivalence classes, see \cref{relation:equiv}, of metrics which correspond to strictly metric systems. Note that, in particular, an upper bound on the number of faces of $\text{MET}_n$ yields such a bound.
Graham, Yao and Yao 
\cite{GYY}
showed that this cone has at most 
$2^{2.72n^2}$ faces.
Inspired by their ideas, 
we improve their bounds.
Recall that $\Delta_{i,j;k} = e_{i,k} + e_{k,j} - e_{i,j}$ is a vector in $\mathbb{R}^{\binom{[n]}{2}}$. For all partial functions
$f:\binom{[n]}{2}\to [n]$ we always assume that $f(a,b)\notin \{a,b\}$
whenever $\{a,b\} \in \text{dom} f$.
We introduce the following definition:
\begin{definition}
The signature $\sigma(f)\in \mathbb{Z}^{\binom{[n]}{2}}$
of a partial function $f:\binom{[n]}{2}\to [n]$ is the vector defined by
\[\sigma(f)= \sum_{\{i,j\} \in \text{dom}(f) } \Delta_{i,j; f(i,j)}.\]
\end{definition}
The following lemma shows that signatures
separate résumés of strictly metric
path system from non-résumés:
\begin{lemma}\label[lemma]{lem:sig_unique}
Let $\mathcal{P}$ be a strictly metric path system and let $f,g:\binom{[n]}{2} \to [n]$, where
$f\in \mathcal{R}(\mathcal{P})$ and $g\notin \mathcal{R}(\mathcal{P})$.
Then $\sigma(f)\neq \sigma(g)$.
\end{lemma}
\begin{proof}
The domain size of a function can be 
easily recovered from its signature:
\[|\text{dom}(f)| = \sum_{\{i,j\} \in \text{dom}(f) } \langle \Delta_{i,j; f(i,j)},\mathbf{1}\rangle =\langle\sigma(f),\mathbf{1}\rangle,\] 
where $\mathbf{1}$ is the all-ones 
vector. Arguing by contradiction, we
suppose that $f$ and $g$ are as in the
theorem and yet $\sigma(f)=\sigma(g)$.
As we saw, this implies that 
$|\text{dom}(g)| = |\text{dom}(f)|$.

We find next $\{a,b\}\in \text{dom}(g)$ 
such that 
$\{a,b;g(a,b)\} \notin T(\mathcal{P})$.
\begin{itemize}
\item
If $\text{dom}(g) \neq \text{dom}(f)$ then we can take any 
$\{a,b\} \in \text{dom}(g) \setminus \text{dom}(f)$.
\item
If $\text{dom}(g) =\text{dom}(f)$: By assumption
$g\not\in \mathcal{R}(\mathcal{P})$. Therefore,
there must be a pair 
$\{a,b\}\in \text{dom}(g)$ for which 
$g(a,b)\notin P_{a,b}$.
\end{itemize}
Let us set $S = \{\{a,b;c\} \in T(P) : f(a,b) \neq c\}$. We can write
\[\sum_{t\in T(\mathcal{P})}\Delta_t =\sigma(f) + \sum_{t\in S}\Delta_t  =\sigma(g) + \sum_{t\in S}\Delta_t = \sum_{\{u,v\}\in \text{dom}(g)} \Delta_{u,v;g(u,v)}+ \sum_{t\in S}\Delta_t.\]
Since $\{u,v;g(u,v)\}\notin T(\mathcal{P})$ for some $u,v$, by \cref{lem:semi_met_char}, $\mathcal{P}$ is not strictly metric.

\end{proof}
We now prove \cref{thm:strict_count}:
\begin{proof}(\Cref{thm:strict_count})
We start with the lower bound. As
\cref{lem:gnp_systems} shows,
a.a.s.\ $|\mathscr{NS}(G)|\ge 2^{\frac{n^2}{64}(1+o(1))}$,
when $G\sim G(n,1/2)$. There is no
double counting to worry about, since a neighborly
path system uniquely defines its graph. Also, all 
$n$-vertex graphs are equally likely
in $G(n,1/2)$. Consequently,
\[|\mathscr{S}_n| =\sum_{G} |\mathscr{NS}(G)| \ge 2^{\binom{n}{2}(1+o(1))}2^{\frac{n^2}{64}(1+o(1))}=  2^{\frac{33}{64}n^2(1+o(1))}.\]
It is possible to increase the 
coefficient of $\frac{33}{64}$ by
an additive term of about
$3\cdot 10^{-4}$. To this end we
work in $G(n,p)$ for $p$ a little
above $\frac 12$. The function
to be maximized is
$p^2(1-p)/4+H(p)-1$, where $H$ is the
binary entropy function, and the maximum
is attained at $p\approx 0.51037$.\\
As \cref{lem:sig_unique} shows, we can
obtain an upper bound on the number of
strictly metric path systems, by 
bounding the total number of possible 
signatures for all partial functions 
$f:\binom{[n]}{2}\to [n]$. It is more
convenient to work with vectors of the
form $\sigma(f) +\mathbf{1}$. This is a 
vector of non-negative integers, and
following up on a previous calculation, 
we get
\[\sum_{i<j} (\sigma(f) +\mathbf{1})_{i,j} =\langle \sigma(f) +\mathbf{1}, \mathbf{1} \rangle = |\text{dom}(f)| + \binom{n}{2}.\]
We bound the number of possible 
signatures that sum to $m + \binom{n}
{2}$ for some 
$0\le m \le \binom{n}{2}$.
This, in turn, does not exceed the
number of integer solutions to the 
system 
	\[\sum_{e\in \binom{[n]}{2}} x_e = m+\binom{n}{2}, \hspace{15mm} x_e\ge 0,  \  e\in \binom{[n]}{2},\]
which equals 
\[\binom{2\binom{n}{2} + m-1}{\binom{n}{2}-1}\le \binom{3\binom{n}{2}}{\binom{n}{2}}\le 2^{H(\frac{1}{3}) \frac{3}{2}n^2(1+o(1))}\le 2^{1.38n^2(1-o(n))}.\]
	Summing over all $0\le m \le \binom{n}{2}$ we get the desired upper bound. 

\end{proof}
\subsection{Faces of the Metric Cone}
We can use similar methods to bound the number of faces of the metric cone, 
\[\text{MET}_n = \{x\in \mathbb{R}^{\binom{n}{2}} : x_{a,b}+x_{b,c}-x_{a,c} \ge 0, \  \forall a,b,c\in [n]\}.\]
An upper bound of $2^{2.72n^2}$ is mentioned in \cite{GYY} without proof. 
Here, we prove a slightly better upper bound:
\begin{theorem}\label[theorem]{thm:face_bound}
The order-$n$ metric cone, $\text{MET}_n$, has at most $2^{2.38 n^2 (1+o(1))}$ faces.
\end{theorem}
Before we prove this upper bound, we briefly discuss lower bounds for $\text{MET}_n$'s number 
of faces. Since every strictly metric path system uniquely corresponds to some face of the 
metric cone,
the lower bound in \cref{thm:strict_count} provides as well
a lower bound for the number of faces. This is the best lower bound that we are aware of. 
We also mention a result of Avis \cite{Av} that $\text{MET}_n$ has at least 
$2^{\binom{n}{2}(1+o(1))}$ extreme rays. His argument shows that the "standard metric" of
asymptotically almost every graph corresponds to an extreme ray of the metric cone.
By "standard metric" of a graph we mean the one where every edge is of unit length. 

Let us return to the equivalence relation \eqref{relation:equiv}
on $\text{MET}_n$, where two metrics are $\sim$-equivalent
if they have the same collection of collinear triples.
The $\sim$-equivalence classes 
are in a one-to-one correspondence with $\text{MET}_n$'s faces. In fact, each such 
equivalence class coincides with the relative interior of its associated face.
In the interior of $\text{MET}_n$ all triangle inequalities are strict, so 
that $T(\rho)=\emptyset$ for every $\rho \in \text{int}(\text{MET}_n)$. 
This equivalence class corresponds to the full-dimensional face that is the cone's interior. 
At the other extreme is the cone's vertex, namely the origin $\{0\}$, a zero-dimensional face 
where $T(\rho) = \binom{[n]}{3}^\ast$. Every other equivalence class is contained in the 
boundary $\partial(\text{MET}_n)$, and corresponds to a non-trivial face. 
We seek to bound the number of 
these equivalence classes, i.e., the cardinality of the set 
\[\mathcal{F}\coloneqq \{F\subseteq \binom{[n]}{3}^\ast: \exists \rho \in \text{MET}_n \text{ s.t. } T(\rho) = F\}.\]
To get started with our proof, we extend the notion of signature to
arbitrary subsets of $\binom{[n]}{3}^\ast$. Namely, let the signature of 
a set of pointed triples $S\subseteq \binom{[n]}{3}^\ast$ be: 
\[\sigma(S) = \sum_{t\in S} \Delta_t.\]
We also need to define $\text{cl}(S)$, the {\em closure} of a subset 
$S \subseteq \binom{[n]}{3}^\ast$. Namely, 
\[\text{cl}(S) := \bigcap_{\substack{F\in \mathcal{F}\\S\subseteq F} }F.\]

Note that $\text{cl}(S) \in \mathcal{F}$, since
$\mathcal{F}$ is intersection closed. Indeed, if $F_1 =T(\rho_1)$ and $F_2 = T(\rho_2)$, 
then it is not difficult to see that $F_1\cap F_2 = T(\rho_1+\rho_2)$. 
We remark, that much of what we say here is not specific to the metric cone and holds just
as well for any polyhedral cone. The crux of the matter seems to be that the face sets of a 
cone and its dual form anti-isomorphic lattices, see e.g. \cite{SW,Zi}. For example,
consider dual cone of $\text{MET}_n$. Its extreme rays are exactly
the rays spanned by the vectors $\Delta_{s}$, $s\in \binom{[n]}{3}^\ast$.\\
Signature and closure are related as follows:
\begin{lemma}\label[lemma]{lem:sig_and_closure}
If $\sigma(S) = \sigma(\tilde{S})$
for some $S,\tilde{S}\subseteq \binom{[n]}{3}^\ast$,  
then $\text{cl}(S) = \text{cl}(\tilde{S})$.
\end{lemma}
\begin{proof}
Let $\rho,\tilde{\rho}\in \text{MET}_n$ be such that 
$\text{cl}(S) =T(\rho)$, $\text{cl}(\tilde{S}) =T(\tilde{\rho})$. 
By definition, $\langle \rho ,\Delta_s \rangle = 0$ holds for all $s\in S$. In particular,
	\[0 = \sum_{s\in S} \langle \rho, s\rangle = \langle \rho,\sigma(S)\rangle = \langle \rho,\sigma(\tilde{S})\rangle = \sum_{s\in \tilde{S}} \langle \rho, s\rangle.\]
	But $\langle \rho, s\rangle\ge 0$ for all $s\in \binom{[n]}{3}^\ast$, and therefore $\langle \rho, s\rangle=0$ for all $s\in \tilde{S}$. This implies $\tilde{S} \subseteq T(\rho)$ and hence $T(\tilde{\rho}) = \text{cl}(\tilde{S})\subseteq T(\rho)$. By the same argument,  $T(\rho)\subseteq T(\tilde{\rho})$.
\end{proof}
Now we can prove \cref{thm:face_bound}.
\begin{proof}(\Cref{thm:face_bound})
For every $F\in \mathcal{F}$ we fix a subset $B\subseteq F$ such that $\{\Delta_t:t\in B\}$ 
is a basis of the subspace spanned by $\{\Delta_t:t\in F\}$. Note that $\text{cl}(B) = F$. So, by \cref{lem:sig_and_closure} 
all these sets have distinct 
signatures. It therefore suffices to bound the number of such signatures.\\
Clearly, $0\le |B|\le \binom{n}{2}$, since
$B$ is a basis for some subspace of $\mathbb{R}^{\binom{[n]}{2}}$. 
How many distinct signatures 
$\sigma(S)$ are there of sets 
$S\subset \binom{[n]}{3}^\ast$ of 
given cardinality 
$0\le m \le \binom{n}{2}$? 
For every $t\in S$, the vector $\Delta_t$ has two $+1$ entries and one $-1$ entry. 
Therefore, it is possible to express $\sigma(S)$ as $v_p - v_n$, where 
$v_p,v_n$ are ${\binom{[n]}{2}}$-dimensional vectors of non-negative integer
entries such that $\langle v_n, \mathbf{1}\rangle = m$, and 
$\langle v_p, \mathbf{1}\rangle = 2m$. 
The number of such vectors $v_p$ does not exceed the number of integer solutions to the system 
\[\sum_{e\in \binom{[n]}{2}}x_e= 2m, ~~~ x_e\ge 0,~~~ \forall e\in \binom{[n]}{2},\]
which equals 
\[\binom{2m + \binom{n}{2}-1}{\binom{n}{2}-1}\le \binom{3\binom{n}{2}}{\binom{n}{2}} \le 2^{H(\frac{1}{3})\frac{3}{2}n^2(1+o(1))}.\]
Similarly, the number of vectors $v_n$ is bounded by
\[\binom{m + \binom{n}{2}-1}{\binom{n}{2}-1}\le \binom{2\binom{n}{2}}{\binom{n}{2}} \le 2^{n^2(1+o(1))}.\] 
Therefore, the number of signatures $\sigma(S)$ with $|S|=m$ is bounded by 
\[\left[2^{H(\frac{1}{3})\frac{3}{2}n^2(1+o(1))}\right] \left[2^{n^2(1+o(1))}\right] \le 2^{2.38n^2(1+o(1))}.\]
	Summing over $0\le m \le \binom{n}{2}$ yields the result.
\end{proof}

\subsection{Strictly Metric Systems in a Fixed Graph}
As before, $\mathscr{NS}(G)$ denotes the number
of neighborly strictly metric path systems of a given graph $G$.
We seek the maximum of 
$|\mathscr{NS}(G)|$ over all
$n$-vertex graphs $G$. As 
\cref{thm:strict_count} and 
\cref{lem:gnp_systems} show, $|\mathscr{NS}(G)|=2^{O(n^2)}$ 
for every $n$-vertex graph, yet
$|\mathscr{NS}(G)|=2^{\Theta(n^2)}$ 
for asymptotically almost all of them. 
We ask for the largest $\alpha>0$ 
for which there exist $n$-vertex graphs with
$|\mathscr{NS}(G)| \ge 2^{\alpha n^2}$.
As it turns out, balanced complete 
bipartite graphs yield 
$\alpha\ge\frac 18$.
\begin{theorem}\label[theorem]{thm:bip_strict}
There holds 
$|\mathscr{NS}(K_{\frac n2,\frac n2})| \ge 2^{\frac{n^2}{8}}$.
\end{theorem}
\begin{proof}
This proof can be viewed as a
simpler version of the proof of
\cref{lem:gnp_systems}. 
Let $X=\{x_1,\dots,x_{\frac n2}\}$ and 
$Y=\{y_1,\dots, y_{\frac n2}\}$ be the 
two parts of $K_{\frac n2,\frac n2}$'s 
vertex set. We introduce the perfect matching 
$M$ of all edges $x_iy_i$ for 
$i=1,\ldots, \frac n2$, and
think of the edges in $M$ as `cheap', so that
shortest paths tend to use them. 
For every $1\le i < j \le \frac n2$, there are precisely two $2$-hop paths
between $x_i$ and $x_j$ that use edges from $M$,
namely, $x_iy_ix_j$ and $x_iy_jx_j$. For every pair $x_i,x_j$ we may (independently) choose between one of these two paths. We associate
with every edge $x_iy_j$ 
an i.i.d.\ random variable $X_{i,j}$ that is uniformly distributed on $[0,\epsilon]$, for
some small constant $\epsilon>0$. 
These random variables act as random noise to make every shortest path unique. Define the following weight function $w:E\to \mathbb{R}_{>0}$ on $K_{\frac n2,\frac n2}$:
	\[w(x_iy_j) \coloneqq \begin{cases}
		1 + X_{i,j}& i = j \\
		1.1+ X_{i,j} & \text{if } x_iy_j \text{ is in some chosen path} \\
		1.2 + X_{i,j}& \text{Otherwise}
	\end{cases}\]
It's not difficult to see that every chosen path is uniquely shortest w.r.t.\ these edge weights. There are $\binom{\frac n2}{2}$ pairs $x_i,x_j$ and $2$ choices for each
such pair. The conclusion follows.
\end{proof}

\subsubsection{Monotone path systems}\label{sec:mono}
We now consider a family of neighborly diameter-$2$ path systems $\mathcal{P}$
in $\mathcal{J}_n$, the join of a 
clique and an anti-clique.
To wit, $V(\mathcal{J}_n)$ 
is the disjoint union of the sets 
$\{x_1,\ldots,x_{n}\}$ and
$\{y_1,\ldots,y_{n}\}$, and the only
pairs of non-neighbors are
$x_i, x_j$ with $1\le i< j\le n$. 
The path system $\mathcal{P}$ that we
are considering is neighborly and diameter $2$, so we
only need to specify, for every
pair $1\le i< j\le n$, the index
$1\le k\le n$, such that
\begin{equation}\label{eq:wohin}
P_{x_i x_j}=x_i y_k x_j.\end{equation}

It is convenient to represent the
(unique) résumé of $\mathcal{P}$
as a symmetric
matrix $M$ where $m_{ij}=m_{ji}=k$
when \cref{eq:wohin} holds
(and $m_{ii}=\ast$ for all $i$).
We call $\mathcal{P}$ {\em monotone}
if the entries in every row, ignoring the diagonal entry, are non-decreasing. By symmetry, the same
holds when we go down columns. In terms of résumés, if $\mathcal{P}$ is monotone and say, $f(x_i, x_j) = y_a$ and $f(x_i, x_k) = y_b$ then $j\le k$ implies $a\le b$.

\begin{theorem}\label[theorem]{thm:monotone}
Every monotone path system
as above, is strictly metric.
\end{theorem}

We start with some technical comments 
on consistent diameter-$2$ path 
systems. 
Below we use the convention that $\alpha_{a,a;c} = \alpha_{a,b;a} = 0$.
\begin{lemma}\label[lemma]{lem:diam_2_inequalities}
Let $f:\binom{[n]}{2}\to [n]$ be the
résumé of a diameter-$2$, consistent
$n$-vertex path system $\mathcal{P}$. Then $\mathcal{P}$ is strictly metric if and only if there exists $x\in \mathbb{R}^{\binom{[n]}{2}}$ satisfying the following system of linear inequalities,
\begin{equation}\label{eqn:sm}
x_{a,f(a,b)} + x_{f(a,b),b} < x_{a,c} + x_{c,b} \hspace{10mm} \forall a,b \in \text{dom}(f), c\neq f(a,b).
	\end{equation}
\end{lemma}
Before we go into the proof, let us
emphasize that we {\em do not limit
ourselves to non-negative $x$}, as the
proof amply illustrates.
\begin{proof}(\Cref{lem:diam_2_inequalities})
Clearly, if $\mathcal{P}$ is strictly metric, then the system (\ref{eqn:sm}) is feasible,
even with positive $x$.  

Next we make some comments on feasible solutions to the system (\ref{eqn:sm}).
We express such solutions in terms of $w\in \mathbb{R}^{\binom{[n]}{2}}$, and assume
w.l.o.g.\ that $w_{u,v} =\infty$ for every edge
$uv$ that does not appear in $\mathcal{P}$. We also
observe that feasible solutions are
translation invariant. Namely, for any real constant $K$, the function
$w$ is a feasible solution if and only if $w+K\mathbf{1}$ is. 

To proceed, let us choose $K>0$ much bigger than $|w_{u,v}|$  
for every edge $uv\in \mathcal{P}$.
We argue that all paths in $\mathcal{P}$ are uniquely shortest on the complete graph with edge weights $\tilde{w} = w+K\mathbf{1}$.
For  $P_{u,v} \in \mathcal{P}$ we need to show that every competitor $uv$-path $Q$ has strictly larger weight. We can clearly assume that $Q$ uses only edges in $\mathcal{P}$ or else $\tilde{w}(Q) = \infty$. If $P_{u,v}$ is an edge, then necessarily any other $uv$-path $Q$ has $2$ or more edges, whence 
$\tilde{w}(P_{u,v}) = K + O_K(1) < 2 K + O_K(1) < \tilde{w}(Q) + O_K(1)$.

If $P_{u,v}$ is a $2$-edge path,
then $uv\notin \mathcal{P}$ and so $\tilde{w}(uv)=\infty$. 
If $Q$ is a $uv$-path of $3$ or more edges, then 
$\tilde{w}(P_{u,v}) = 2K + O_K(1) < 3 K + O_K(1) < \tilde{w}(Q) + O_K(1)$. 
Finally, if $Q = uzv$ is another $uv$-path of $2$ edges then 
	\[\tilde{w}(P_{u,v}) = \tilde{w}_{u,f(u,v)} + \tilde{w}_{f(u,v),v} < \tilde{w}_{u,z} + \tilde{w}_{z,v}  =\tilde{w}(Q),\]
	since $\tilde{w}$ is a feasible solution of the above system of inequalities.
\end{proof}
\Cref{lem:diam_2_inequalities} yields
the following result which we use to
prove \cref{thm:monotone}. It is also possible to accomplish this using \cref{lem:semi_met_char}.
\begin{lemma}\label[lemma]{lem:diam_2_witness}
Let $f:\binom{[n]}{2}\to [n]$ be the
résumé of a diameter-$2$, consistent
$n$-vertex path system $\mathcal{P}$. Then $\mathcal{P}$ is not strictly metric if and only if there exist
coefficients $\alpha_{a,b;c}\ge 0$ such that for all $a,b\in \text{dom}(f)$, $\sum_{i=1}^{n}\alpha_{a,b;i} = 1$ and 
	\[\sum_{a,b \in \text{dom}(f)} (e_{a,f(a,b)} + e_{f(a,b),b}) = \sum_{a,b \in \text{dom}(f)} \sum_{i=1}^{n}\alpha_{a,b;i} (e_{a,i} + e_{i,b})\]
	and $\alpha_{a,b;f(a,b)} < 1$ for some $a,b\in \text{dom}(f)$.
\end{lemma}
\begin{proof}
This claim follows from \cref{lem:diam_2_witness} and some
form of
LP duality, such as Farkas Lemma or Gordan's Lemma \cite{Dan}. Indeed, the system
\begin{equation*}
x_{a,f(a,b)} + x_{f(a,b),b} < x_{a,c} + x_{c,b} \hspace{10mm} \forall a,b \in \text{dom}(f), c\notin  \{a,b, f(a,b)\}
	\end{equation*}
	is not feasible if and only if there exists $\beta_{a,b;c}\ge 0$ not all zero such that 
	\[\sum_{a,b\in \text{dom}(f)}\sum_{i=1}^{n}\beta_{a,b;i}(e_{a,f(a,b)} + e_{f(a,b),b}) = \sum_{a,b\in \text{dom}(f)}\sum_{i=1}^{n}\beta_{a,b;i}(e_{a,i} + e_{i,b}),\]
    where we set $\beta_{a,b;f(a,b)} = 0$.
	By scaling if necessary, we can assume that $\sum_{i=1}^{n}\beta_{a,b;i}<1$ for all $a,b$. Adding $\sum_{a,b\in \text{dom}(f)} \left(1-\sum_{i=1}^{n}\beta_{a,b;i}\right)(e_{a,f(a,b)} + e_{f(a,b),b})$ to both sides we get 
	\[\sum_{a,b\in \text{dom}(f)}(e_{a,f(a,b)} + e_{f(a,b),b}) = \sum_{a,b\in \text{dom}(f)}\sum_{i=1}^{n}\alpha_{a,b;i}(e_{a,i} + e_{i,b}),\]
	where 
	\[\alpha_{a,b;i} =\begin{cases}
		\beta_{a,b;i} & i\neq f(a,b) \\
		1-\sum_{i=1}^{n}\beta_{a,b;i} & i = f(a,b)
	\end{cases}\]
	and $\alpha_{a,b;f(a,b)} < 1$ for at least one $a,b$.

\end{proof}
Now we prove \cref{thm:monotone}
\begin{proof}
In this proof we also refer to the (monotone) matrix $M$
mentioned after \eqref{eq:wohin}, which may help the reader's intuition.\\ 
If $\mathcal{P}$ is not strictly metric, then by \cref{lem:diam_2_witness}, this is witnessed 
by some non-negative $\alpha_{x_i,x_j;y_k}$ for which 
$\sum_{k=1}^{n}\alpha_{x_i,x_j;y_k}=1$ and
\[\sum_{i<j} (e_{x_i,f(x_i,x_j)} + e_{x_j, f(x_i,x_j)}) = \sum_{i<j} \sum_{k=1}^{n}\alpha_{x_i,x_j;y_k} (e_{x_i,y_k} + e_{x_j,y_k})\]
and $\alpha_{x_i,x_j;f(x_i,x_j)} < 1$ for some $i<j$.
Notice that the coefficient of $e_{x_a,y_c}$
counts the paths in $\mathcal{P}$ 
which start at $x_a$ and go via $y_c$. In particular, for fixed $x_a$ and $y_c$
\[\left|\{j\ge 1 : f(x_a, x_j) = y_c \}\right| = \sum_{j =1}^{n} \alpha_{x_a,x_j;y_c}.\]
To simplify our notation, we write below
$\{a,b;c\}$ rather than
$\{x_a,x_b;y_c\}$ and
$f(a,b) = c$ instead of   
$f(x_a,x_b) = y_c$. \\
Let $(a,b)$ with $a<b$ be the lexicographically first
pair of indices such that $\alpha_{a,b;f(a,b)} < 1$.
As discussed above, by \cref{lem:diam_2_witness} we have
\[\left|\{j\ge 1 : f(a, j) = f(a,b) \}\right| = \sum_{j =1}^{n} \alpha_{a,j;f(a,b)},\]since
both sides count the occurrences of $f(a,b)$ in row $a$ of $M$.\\
By the minimality assumption, $\alpha_{a,j;f(a,j)} = 1$ for  $j < b$ and $\alpha_{a,j;k} = 0$ whenever 
 $j < b$, $k\neq f(a,j)$. So we can write
\[\left|\{j \ge b : f(a, j) = f(a,b) \}\right| = \sum_{j =b}^{n} \alpha_{a,j;f(a,b)},\]
where both sides count the occurrences of $f(a,b)$ in row $a$ of $M$
starting from column $b$.\\
Since $M$ is monotone, if two entries in the same row are
equal, then all entries in between them are equal to them as well.\\
In other words, there exists some $c > b$ such that $f(a,j) = f(a,b)$ 
for all $b\le j < c$ and $f(a,j) >  f(a,b)$ for $j\ge c$. 
We can therefore write
\[c - b =\left|\{j\ge b : f(a, j) = f(a,b) \}\right| = \sum_{j =b}^{c-1} \alpha_{a,j;f(a,b)} +\sum_{j =c}^{n} \alpha_{a,j;f(a,b)}.\]
Since $\alpha_{i,j;k} \le 1$ and $\alpha_{a,b;f(a,b)}<1$ we have $\sum_{j=b}^{c-1} \alpha_{a,j;f(a,b)} < c-b$. In particular, there is some $l\ge c$ such that $\alpha_{a,l;f(a,b)} > 0$. Now, it must also be the case that 
\[\left|\{i \ge 1 : f(i, l) = f(a,b) \}\right| = \sum_{i=1}^{n}\alpha_{i,l,f(a,b)}.\]
By the minimality of $a$ we have that $\alpha_{i,l,f(i,l)} = 1$ for $i<a$ and $\alpha_{i,l,c} = 0$ for $i<a$, $c\neq f(i,l)$. Therefore,
\[\left|i\ge a : f(i,l) = f(a,b) \right| = \sum_{i=a}^{n}\alpha_{i,l,f(a,b)}.\]
As before, both sides count the occurrences of $f(a,b)$ in column $l$ starting from row $a$.
On the other hand, since $f(a,l) > f(a,b)$, by monotonicity $f(i,l) \ge f(a,l)> f(a,b)$ for all $i\ge a$. Therefore,
\[0=\left|i\ge a : f(i,l) = f(a,b) \right|  = \sum_{i=a}^{n}\alpha_{i,l,f(a,b)} \ge \alpha_{a,l,f(a,b)}>0,\]
a contradiction.
\end{proof}
Monotone path systems yield a slight improvement of our previous bound.
\begin{proposition}\label[proposition]{prop:monotone_count}
The number of monotone path systems in
the $n$-vertex graph $\mathcal{J}_{\frac{n}{2}}$ is $2^{\alpha n^2(1+o(1))}$.  
In particular,
\[|\mathscr{NS}(\mathcal{J}_{\frac{n}{2}})|\ge 2^{\alpha n^2(1+o(1))},\] where $\alpha = \frac{3}{16}\left(3\log_23 - 4\right)\approx .141$
\end{proposition}
The proof of \cref{prop:monotone_count} builds on the notion of {\em plane partitions}.
These generalization of integer partitions were introduced in \cite{Mac}. 
Let us recall some of the essentials.
A plane partition of the positive integer $n$ is a two-dimensional array of non-negative 
integers $\pi_{i,j}$ that is non-increasing in both indices and which sums up to $n$, i.e. $\pi_{i+1,j} \le \pi_{i,j}$, $\pi_{i,j+1}\le \pi_{i,j}$ and $\sum_{i,j}\pi_{i,j} = n$. 
An $(r,s,t)$-{\em boxed plane partition} is a plane partition such that $\pi_{i,j}\le t$ with $\pi_{i,j}$ non-zero only when $i\le r$ and $j\le s$. In other words, it is an $r\times s$ matrix with entries in $\{0,1\dots,t\}$ 
whose rows and columns are non-increasing. For example, here is a $(3,4,5)$-boxed plane partition of $21$:
\[\begin{matrix}
	5 & 3 & 3 & 1\\
	4 & 2 & 1 & 0\\
	2 & 0 & 0 & 0
\end{matrix}.\]
MacMahon \cite{Mac} found a formula for 
$N(r,s,t)$, the number of $(r,s,t)$-boxed plane partitions.
\begin{proposition}\label[proposition]{prop:macmahon_formula}
	For all $r,s,t\ge 1$ we have 
	\[N(r,s,t) = \prod_{i=1}^{r}\prod_{j=1}^{s} \frac{i+j+t-1}{i+j-1}.\]
\end{proposition}
A {\em symmetric} $(r,t)$ plane partitions is an $(r,r,t)$-boxed plane partition s.t.\ $\pi_{i,j} = \pi_{j,i}$ for all $i$ and $j$.
In other words, this is a symmetric $r\times r$ matrix with entries in $\{0,1,\dots,t\}$ with non-increasing rows (and columns). Proving a conjecture of MacMahon, Andrews \cite{An} found a formula for  
$S(r,t)$, the number of symmetric $(r,t)$ plane partitions.
\begin{proposition}\label[proposition]{prop:macmahon_sym_formula}
	For all $r,t\ge 1$ we have 
	\[S(r,t) = \prod_{i=1}^{r}\frac{2i+t-1}{2i-1}\prod_{1\le i < j \le r} \frac{i+j+t-1}{i+j-1}.\]
\end{proposition}
Let us return to path systems. A matrix $M$ as in \cref{sec:mono} is an
$(n,n-1)$ symmetric plane partition, with the only restriction that $m_{ii}=\ast$
for all $i$. This allows us to reduce the counting problem of such monotone systems
to the asymptotics of $S(n,n)$, because the additional factor that comes from all the possible
diagonals of $M$ is only a lower order term.

\begin{lemma}\label[lemma]{lem:macmahon_asymptotics}
We have \[N(n,n,n) = \left(\frac{27}{16}\right)^{\frac{3}{2}n^2(1+o(1))}\hspace{10mm}\text{ and } \hspace{10mm} S(n,n) = \left(\frac{27}{16}\right)^{\frac{3}{4}n^2(1+o(1))}.\] 
\end{lemma}
\begin{proof}
Put differently, we are claiming that
\[\lim_{n\to \infty} \frac{\ln N(n,n,n)}{n^2} = 2\lim_{n\to \infty} \frac{\ln S(n,n)}{n^2} = 2\beta,\]
where $\beta = \frac{3}{4}\left(3\ln 3 - 4\ln 2\right)$.
By propositions \ref{prop:macmahon_formula} and \ref{prop:macmahon_sym_formula} we have 
\[S(r,t)^2 = N(r,r,t)\prod_{i=1}^{r} \frac{2i+t-1}{2i-1}, \]
and it follows that $2\lim_{n\to \infty} \frac{\ln S(n,n)}{n^2}=\lim_{n\to \infty} \frac{\ln N(n,n,n)}{n^2}$. Therefore, it suffices to prove that the latter limit exists and is equal to $2\beta$. First we can rewrite \cref{prop:macmahon_formula} as 
\[N(r,s,t) = \prod_{i=1}^{r}\prod_{j=1}^{s} \frac{i+j+t-1}{i+j-1} = \prod_{k=1}^{t} \frac{\binom{r+s+k-1}{s}}{\binom{s+k-1}{s}.}\]
Therefore, 
	\begin{equation*}
		\begin{split}
			\ln N(n,n,n) & = \sum_{k=1}^{n} \ln\binom{2n+k-1}{n} - \ln \binom{n+k-1}{n} \\
			& =(1+o(1)) \sum_{k=1}^{n} \left[(2n+k)\ln(2n+k) -2(n+k)\ln(n+k) + k\ln k\right],
		\end{split}
	\end{equation*}
where in the second line we have applied Stirling's approximation. For every $a\ge 0$, the
function $f_a(x) = (a+x)\ln(a+x)$ is increasing with $x\ge 1$, and therefore 
\[\sum_{k=1}^{n} f_a(k) \le \int_{1}^{n} f_a(x) dx \le \sum_{k=2}^{n+1} f_a(k).\]
In particular,
\[\ln N(n,n,n) \le \int_{1}^{n} \left[f_{2n}(x)  - 2 f_{n}(x) + f_{0}(x) \right]dx \le \ln N(n,n,n) + O(n\ln n).\]
We turn to determine this integral, starting with the indefinite one
$$\int f_a(x)dx=\int (a+x)\ln(a+x) dx= \frac{1}{2}(a+x)^2\ln(a+x) - \frac{(a+x)^2}{4}$$ 
and hence the definite integral yields
\[\int_{1}^{n} \left[f_{2n}(x)  - 2 f_{n}(x) + f_{0}(x) \right]dx = \left(\frac{9}{2}\ln 3 - 6 \ln 2\right)n^2 + O(n).\]
It follows that 
\[\lim_{n\to \infty} \frac{N(n,n,n)}{n^2} =\frac{9}{2}\ln 3 - 6 \ln 2.\]
\end{proof}
Finally, we prove \cref{prop:monotone_count}.
\begin{proof}(\Cref{prop:monotone_count})
For ease of notation we prove the claim for the $2n$-vertex graph $\mathcal{J}_n$, rather
than the $n$-vertex $\mathcal{J}_{\frac{n}{2}}$. As already mentioned, this amounts to showing
that after taking the log, the number of monotone systems in $\mathcal{J}_n$ and the number of $(n,n)$ symmetric plane partitions agree up to $1+o_n(1)$ factor. Let $\mathcal{M}$ denote the set of monotone path systems in $\mathcal{J}_n$. There is an injection from $\mathcal{M}$ to set of symmetric $(n,n-1)$ plane partition. Indeed, if $f:\binom{[n]}{2}\to [n]$ is the résumé of a monotone system we associate with
it a plane partition $\pi_{i,j}$ such that $n-f(i,j) = \pi_{i,j} = \pi_{j,i}$ and $\pi_{i,i} = \pi_{i+1,i}$. It follows that $|\mathcal{M}| \le S(n,n-1)$. On the other hand, given a symmetric $(n,n-1)$ plane partition $\pi_{i,j}$, the résumé $f(i,j) = n-\pi_{i,j}$ defines a monotone path system. The only reason that this mapping is not injective is that modifications of the
diagonal entries $\pi_{i,i}$ leave the induced path system unchanged. Therefore, 
we arrive at the inequality $S(n,n-1) \le |\mathcal{M}|n^{n}$. Observe that 
\[\lim_{n\to \infty} \frac{\ln |\mathcal{M}|}{n^2} \le  \lim_{n\to \infty} \frac{\ln S(n,n-1)}{n^2} = \lim_{n\to \infty} \frac{\ln S(n,n)}{n^2}\]
and 
\[\lim_{n\to \infty} \frac{\ln |\mathcal{M}|}{n^2} \ge  \lim_{n\to \infty} \frac{\ln (n^{n}S(n,n-1))}{n^2} = \lim_{n\to \infty} \frac{\ln S(n,n)}{n^2}.\]
The claim then follows from \cref{lem:macmahon_asymptotics}.
\end{proof}

\begin{remark}
The constant in \cref{prop:monotone_count} can be modestly improved by considering 
$\mathcal{B}_{n,\gamma}$, a slightly different $n$-vertex graph. It is the join of an anti-clique of size $\beta n$ with a clique of size $(1-\gamma)n$, where $0 <\gamma < 1$. An identical argument shows that monotone systems in such a graph are also strictly metric. The
analysis is similar and shows that $\gamma \approx 0.63$ maximizes the number of monotone systems and yields $\mathscr{NS}(\mathcal{B}_{n,\gamma}) \ge 2^{\xi n^2 (1+o(1))}$ where $\xi \approx 0.154$.
\end{remark}

\section{Maximum VC Classes}
\label{sec:vc}
We now take a different perspective of path 
systems provided by the theory of 
VC dimension. We first recall the 
fundamentals of this domain. Let
$\mathcal{F} \subseteq 2^{[n]}$ be a set 
system for some $n\ge 1$.  
We say that $\mathcal{F}$ {\em shatters} a set $S\subseteq [n]$ if 
$\mathcal{F}\cap S \coloneqq \{S\cap F:F\in\mathcal{F}\} = 2^S$.
The {\em VC dimension} of $\mathcal{F}$ 
is \[\dim_{VC} \mathcal{F} = \max\{|S|: S\subseteq [n], \text{ $\mathcal{F}$ shatters $S$}\}.\]
The well known Perles-Sauer-Shelah bound (\cite{Sh}) states that if $\mathcal{F} \subseteq 2^{[n]}$ 
and $\dim_{VC} \mathcal{F}=d$, then $|\mathcal{F}|\le \sum_{i=0}^{d}\binom{n}{i}$.
In this view, a class $\mathcal{F} \subseteq 2^{[n]}$ of VC dimension $d$ 
is said to be a {\em maximum class} if $|\mathcal{F}| =  \sum_{i=0}^{d}\binom{n}{d}$. \\
Let $\mathcal{P}$ be a consistent path system on $[n]$ and let $\mathcal{F}\subset 2^{[n]}$,
\[\mathcal{F}_{\mathcal{P}} \coloneqq \{V(P) : P\in \mathcal{P}\} \cup \{\emptyset, \{1\},\dots \{n\}\}.\]
In other words, $\mathcal{F}_{\mathcal{P}}$ is a set system containing all the singletons,
the empty-set (i.e., a full 1-dimensional skeleton) as well as 
the vertex set of each path in $\mathcal{P}$. We have the following observation:
\begin{observation}\label[observation]{obs:p_system_VC}
For any consistent path system $\mathcal{P}$, the set system $\mathcal{F}_{\mathcal{P}}$ is an intersection-closed maximum class of VC dimension $2$. 
\end{observation}
\begin{proof}
That $\mathcal{F}_{\mathcal{P}}$ is intersection closed follows from the consistency of $\mathcal{P}$. Also, it is clear that $|\mathcal{F}_{\mathcal{P}}| = \binom{n}{2}+\binom{n}{1}+\binom{n}{0}$.
It remains to show that this class has VC dimension $2$. Indeed, $\mathcal{F}_{\mathcal{P}}$ shatters 
{\em every} set $\{u,v\}$ of cardinality 
$2$, because $V(P_{u,v})$, $\{v\}$, 
$\{u\}$ and $\emptyset$ are all contained 
in $\mathcal{F}_{\mathcal{P}}$. 

Finally, $\mathcal{F}_{\mathcal{P}}$ cannot shatter
any $3$-vertex set $S=\{a,b,c\}$. Otherwise, some path $P\in \mathcal{P}$
contains all three vertices. Say they
appear in $P$ in the order $a,b,c$.
By consistency, any path in $\mathcal{P}$
that contains the vertices $a$ and $c$ must also contain $b$, contrary to the
condition of shattering. 
\end{proof}

Let $m(n,d)$ denote the number of maximum classes $\mathcal{F} \subseteq 2^{[n]}$ of VC-dimension $d$. 
The following result is due to Alon, Moran and Yehudayoff \cite{AMY}:
\begin{proposition}\label[proposition]{prop:amy}
For fixed $d \ge 1$,
\[n^{\frac{1}{d+1}\binom{n}{d}(1+o(1))}\le m(n,d) \le n^{\binom{n}{d}(1+o(1))}.\]
\end{proposition}
As shown in \cite{BMW},
the upper bound in \cref{prop:amy} tight. Their lower bound on $m(n,d)$
is attained by appealing to certain induced matchings of the boolean cube. 
Here, we reprove this lower bound from a different perspective, by generalizing the path system construction in \cref{thm:path_system_count}. Explicitly, the statement is:
\begin{theorem}\label[theorem]{thm:max_VC_classes}
For fixed $d\ge 1$,
\[m(n,d) = n^{\binom{n}{d}(1-o(1))}.\]
Moreover, the same holds even for the number of 
intersection closed maximum classes of VC dimension $d$.
\end{theorem}
Before proving \cref{thm:max_VC_classes} we consider the special case when $d=2$. Note that the upper bounds in \cref{prop:amy} along with \cref{obs:p_system_VC} 
reprove the upper bound in \cref{thm:path_system_count}. Moreover, for $d=2$ \cref{thm:max_VC_classes}  follows from the lower bound in
\cref{thm:path_system_count}, and the
fact that $|\mathscr{P}_n|\le m(n,2)$.
Recall that in the lower bound proof of \cref{thm:path_system_count} we considered the number of diameter-$2$ path systems in random $G(n,p)$
graphs. Likewise, in the proof of \cref{thm:max_VC_classes} we apply a
similar idea to {\em Linial-Meshulam random complexes}, \cite{LM,MW}.
Let us recall the fundamentals of this
randomized model. $Y_k(n,p)$ is a
distribution over 
$k$-dimensional simplicial complexes
with vertex set $[n]$.
A complex $Y\sim Y_k(n,p)$ sampled
from this distribution has a full
$(k-1)$-dimensional skeleton, i.e.,
$Y$ contains all subsets of $[n]$ of 
size $\le k$. Every subset of size 
$k+1$ (a "$k$-dimensional face") is 
placed in $Y$ independently with 
probability $p$. The set of $Y$'s
$k$-dimensional faces is denoted by $E(Y)$.
Note that $Y_1(n,p)$
is identical with $G(n,p)$.
Here we'll work with $k=d-1$.

Let $Y$ be a $k$-dimensional simplicial
complex on vertex set $[n]$ and let 
$S\in \binom{[n]}{k+1}\setminus Y$. 
We say that a vertex $a\not\in S$ is 
$Y$-{\em compatible} to $S$ if
$S\cup \{a\}\setminus\{x\} \in Y$
for every $x\in S$. We call
$S\cup\{a\}$ a {\em $Y$-compatible 
extension of the base} $S$. Notice
that the extended set $S\cup\{a\}$ 
indeed uniquely defines its base $S$,
because for any $x\neq a$ in 
$S\cup\{a\}$, the set 
$S\cup\{a\}\setminus\{x\}$ 
belongs to $Y$.\\

The case $k=1$ provides an interesting 
illustration. Here $Y$ is a graph, and 
$S=\{u,v\}$ comprises a pair of
$Y$-non-neighbors. A vertex $a\neq u,v$
is $Y$-{\em compatible} to $S$, if it is
a $Y$-neighbor of both $u$ and $v$.\\
Standard concentration results yield: 
\begin{lemma}\label[lemma]{lem:concentration}
Let $Y\sim Y_k(n,p)$ for some $k\ge 0$.
If
$p \ge \Omega(\frac{1}{\log n})$ then a.a.s.\
\begin{enumerate}
\item $|E(Y)| = p\binom{n}{k+1}(1+o(1))$
\item 
The number of $Y$-compatible vertices  
to any 
$S\in \binom{[n]}{k+1}\setminus Y$
is $p^{k+1}n(1+o(1))$.
\end{enumerate}
\end{lemma}
We now prove \cref{thm:max_VC_classes}.

\begin{proof}
(\Cref{thm:max_VC_classes})

As already mentioned, we wish to show that the upper bound in \cref{prop:amy}
is tight.
We construct a large collection of
maximal families $\mathcal{F}$ of
VC-dimension $d$. In particular, each
$\mathcal{F}$ is a set systems
with $\sum_{i=0}^{d} \binom{n}{i}$ 
sets. To start, we sample 
$Y\sim Y_{d-1}(n,p)$ with 
$p=\frac{1}{\log n}$. We may and will
assume (with asymptotic almost 
certainty) that $Y$ satisfies
the conclusions
of \cref{lem:concentration}.
The family $\mathcal{F}$ comprises
all the faces of $Y$, as well as a
$Y$-compatible extension of every
set in $\binom{[n]}{d} \setminus E(Y)$.
Clearly, $\mathcal{F}$ contains
$\sum_{i=0}^{d} \binom{n}{i}$ sets, as
planned. 

By \cref{lem:concentration}, any set in 
$\binom{[n]}{d}\setminus E(Y)$
has $p^{d} n(1+o(1))$
$Y$-compatible extensions. Since $|\binom{[n]}{d} \setminus E(Y)| = (1-p)\binom{n}{d}(1+o(1))$ we get a total of 
\[\left(p^d n(1+o(1))\right)^{(1-p)\binom{n}{d}(1+o(1))}\]
distinct set systems $\mathcal{F}$ 
of the desired cardinality. 
Moreover, $p = \frac{1}{\log n} = o(1)$ and 
$p^d = n^{-d\frac{\log \log n}{\log n}} = n^{-o(1)}$ and the above expression simplifies to 
\[ n^{\binom{n}{d}(1+o(1))}.\]
There is no double-counting here, 
because, as mentioned, the extended sets
in $\mathcal{F}$ uniquely split to
base-plus-extension.
It remains to show that the VC dimension
of each such system is $d$.
It is at least $d$, 
since clearly all members of $E(Y)$  
are shattered. On the other hands, 
suppose that one such system 
$\mathcal{F}$ shatters some set
$Z$ of cardinality $d+1$. Necessarily,
$Z\in \mathcal{F}$ because all sets
in $\mathcal{F}$ have cardinality
$\le d+1$. By construction, $Z$
must be a $Y$-compatible extension, 
i.e., $Z=S\cup\{a\}$ for some
$S\not\in T$. It follows that 
$\mathcal{F}$ does not shatter $Z$.

\end{proof}

\section{Open Questions}\label{sec:open}
Numerous interesting open questions suggest themselves. Here are some of them
\begin{itemize}
\item 
We have determined the asymptotics of
$\log(|\mathscr{M}_n|)$ as well as
$\log(|\mathscr{P}_n|)$, viz.,
\[ n^{\frac{n^2}{2}(1-o(1))} \le|\mathscr{M}_n|\le |\mathscr{P}_n| \le n^{\frac{n^2}{2}},\]
where $\mathscr{P}_n$ and
$\mathscr{M}_n$ is the set of all consistent, resp.\ 
metric, path systems on $n$ 
vertices. It is reasonable to conjecture
that $|\mathscr{M}_n|=o_n(|\mathscr{P}_n|)$,
and even ask about the rate at
which $|\mathscr{M}_n|/|\mathscr{P}_n|$
tends to zero with $n$.\\
We note that 
$|\mathscr{P}_n\setminus \mathscr{M}_n|=n^{\frac{n^2}{2}(1+o(1))}$.
Recall that the Petersen graph $H$ admits a non-metric neighborly path system 
$\mathcal{P}_H$ \cite{CL}. Let us
fix a copy of $H$ inside a $G(n,\frac{1}{\log n})$ 
graph $G$. Each of the following 
$n^{\frac{n^2}{2}(1+o(1))}$ path systems in $G$ is neighborly,
consistent and non-metric.
On $H$ all these systems coincide with $\mathcal{P}_H$. 
Every other pair of
non-adjacent vertices in $G$ is connected
by any of their $\frac{(1+o(1))n}{\log^2 n}$ length-$2$ paths.
The sub-system $\mathcal{P}_H$ makes all of them non-metric.
\item Almost all bounds in the paper are about the asymptotics of the logarithm.
It is desirable, of course, to upgrade these estimates to direct asymptotics.
\item 
There are various natural measures of 
disagreement $\Delta$ between pairs of 
path systems,
$\mathcal{P}$ and $\mathcal{Q}$ on the
same vertex set $[n]$.
Here are two, but numerous other 
measures suggest themselves:
\begin{enumerate}
\item 
The number of pairs $1\le i < j\le n$
for which $P_{i,j} \neq Q_{i,j}$.
\item 
The number of triples $i, j, k \in [n]$
where $k$ resides on one of the two paths
$P_{i,j}$, $Q_{i,j}$, but not on both.
\end{enumerate}
This induces a notion of distance 
between an $n$-vertex path system 
$\mathcal{P}$ and a class $\mathcal{A}$ of
$n$-vertex path systems, namely
\[\text{dist}(\mathcal{P}, \mathcal{A}):=\min \{\Delta(\mathcal{P}, \mathcal{Q})\} ~|~\mathcal{Q} \in \mathcal{A}\}\]
We have considered various classes of
path systems, e.g., consistent, metric,
strictly metric, neighborly and more.
Given two such classes of path systems
$\mathcal{A}$ and $\mathcal{B}$, it is
interesting to 
\begin{enumerate}
\item 
Determine $\min \text{dist}(\mathcal{P}, \mathcal{A})$ over $\mathcal{P}\in \mathcal{B}$, or
\item 
Understand how $\text{dist}(\mathcal{P}, \mathcal{A})$ is distributed over $\mathcal{P}\in \mathcal{B}$.
\end{enumerate}
\item 
It is particularly intriguing to compare between metric and strictly
metric path systems. Here is a possible perspective on this question.
Define $B$ to be the cube $(0,1)^{\binom{n}{2}}$, where we view a point $w\in B$
as a map $\binom{n}{2}\to (0,1)$. Let $\mathcal{Q}$ be any strictly metric path 
system on vertex set $[n]$. Corresponding to $\mathcal{Q}$ is a set of linear
inequalities in the values $w(i,j)$ that uniquely defines $\mathcal{Q}$.
Namely, it is the intersection of all half spaces that are defined by an inequality
of the form $w(Q_{xy}) < w(Q')$, over all pairs of vertices $x,y$, where $Q_{xy}$
is the $xy$-path in $\mathcal{Q}$ and  $Q'\neq Q_{xy}$ is any other $xy$ path. 
Here, if $P=(v_1, v_2, \ldots)$ is a path, then
$w(P):=\sum_i w(v_i, v_{i+1})$. This collection of linear inequalities defines
a full-dimensional polytope $C_{\mathcal{Q}}$. The collection of polytopes 
$C_{\mathcal{Q}}$, ranging over all strictly metric $n$-point path systems 
$\mathcal{Q}$ partition 
the cube $B$ minus a measure zero set into a finite collection of polytopes.
What do these polytopes typically look like? What are their vertex sets?

This measure-zero set comprises the hyperplane arrangement $\mathcal{A}$ defined
by all homogeneous linear equations of the form $w(P)=w(Q)$
for some two distinct paths $P, Q$ with the same pair of end vertices.

Many questions suggest themselves: Given $w\in \mathcal{A}$, which (non-strict,
but consistent) path systems does it correspond to? How to sort them out?
Can we decide which pairs of strictly metric
path systems share a face? Of which dimension? Which non-strict metric path
systems border on a given strict one? This discussion naturally assigns a dimension
to every non-strict metric path system $\mathcal{P}$, namely, the highest dimension of a face
of $\mathcal{A}$ on which $\mathcal{P}$ resides. What can be said about this dimension?
Does every face of $\mathcal{A}$ contain some non-strict metric path system?
\item 
The Graham, Yao and Yao paper \cite{GYY} does not speak about path systems 
but rather about {\em connection matrices}, for the definition of
which we refer the reader to their paper. The relation between the two
theories is still not fully understood.
\item 
These questions are also closely related to the quest of the
$f$-vector of the metric cone. For this as well as the study of
the metric polytope, see \cite{De,De2,De3}
\item 
For the number of $n$-vertex strictly 
metric path systems we have the estimates:
\[2^{0.51n^2(1-o(1))}\le|\mathscr{S}_n| \le 2^{1.38n^2(1-o(1))}.\]
Is it possible to narrow the gap between
the bounds? Perhaps even derive the exact
constant in the exponent.
\item 
In particular, we still do not know
how many strictly metric path systems
a fixed graph can have. Neither do we know the answer for a random
$G(n,\frac 12)$ graph, i.e., for almost all graphs.
\item 
In \cref{thm:monotone} we find certain classes of monotone path systems that
are strictly metric. We conjecture that this is a special case of a broader phenomenon.
Let $f:\binom{[n]}{2}\to [n]$ be the (unique) résumé of a diameter-$2$ path system 
$\mathcal{P}$. We say $\mathcal{P}$ is monotone if $f(i,j) \le f(i,k)$ whenever 
$\{i,j\}, \{i,k\}\in \text{dom} f$ and $j\le k$. We conjecture:
\begin{conjecture}
Every monotone consistent path system of diameter $2$ is strictly metric.
\end{conjecture}
\item 
We recall that we still do not have a 
satisfactory theory of {\em neighborly}
path systems. To mention one
specific question in this direction: Given
a graph $G$, how hard is it to decide
whether $G$ has a consistent neighborly
path system that is non-metric?
\end{itemize}

\maketitle

\printbibliography

\end{document}